\documentclass[11pt]{amsart}

\usepackage{amssymb,graphicx, amsmath, amsthm, enumitem,MnSymbol, array,mathalpha,enumitem}
\usepackage{calrsfs}
\usepackage{wrapfig}
\usepackage{graphicx}
\graphicspath{ {./images/} }

\DeclareMathAlphabet\mathbfcal{OMS}{cmsy}{b}{n}

\usepackage{todonotes}

\usepackage{floatflt}

\usepackage{floatrow}

\usepackage{hyperref}

\usepackage{mathbbol}

 \usepackage[mathscr]{eucal}

\usepackage[font=small]{caption}

\def\Q{{\mathcal{Q}}} 
\def\k{\mathbb R} 

\def\k{{\kappa}}

\def\f{{f}}

\def\PP{{\mathsf{P}}}

\def\P{{\mathbb{P}}}
\def\mid{{\rm mid}}

\def\B{{\mathbb{B}}}

\begin{document}

\newtheorem{theorem}{Theorem}[section]
\newtheorem{lemma}[theorem]{Lemma}
\newtheorem{proposition}[theorem]{Proposition}
\newtheorem{corollary}[theorem]{Corollary}
\newtheorem{problem}[theorem]{Problem}
\newtheorem{construction}[theorem]{Construction}

\theoremstyle{definition}
\newtheorem{defi}[theorem]{Definitions}
\newtheorem{definition}[theorem]{Definition}
\newtheorem{notation}[theorem]{Notation}
\newtheorem{remark}[theorem]{Remark}
\newtheorem{example}[theorem]{Example}
\newtheorem{question}[theorem]{Question}
\newtheorem{comment}[theorem]{Comment}
\newtheorem{comments}[theorem]{Comments}

\newtheorem{discussion}[theorem]{Discussion}

\renewcommand{\thedefi}{}

\long\def\alert#1{\smallskip{\hskip\parindent\vrule%
\vbox{\advance\hsize-2\parindent\hrule\smallskip\parindent.4\parindent%
\narrower\noindent#1\smallskip\hrule}\vrule\hfill}\smallskip}

\def\ff{\frak}
\def\tf{torsion-free}
\def\Spec{\mbox{\rm Spec }}
\def\Proj{\mbox{\rm Proj }}
\def\hgt{\mbox{\rm ht }}
\def\type{\mbox{ type}}
\def\Hom{\mbox{ Hom}}
\def\rank{\mbox{rank}}
\def\Ext{\mbox{ Ext}}
\def\Tor{\mbox{ Tor}}
\def\ker{\mbox{ Ker }}
\def\Max{\mbox{\rm Max}}
\def\End{\mbox{\rm End}}
\def\xpd{\mbox{\rm xpd}}
\def\Ass{\mbox{\rm Ass}}
\def\emdim{\mbox{\rm emdim}}
\def\epd{\mbox{\rm epd}}
\def\repd{\mbox{\rm rpd}}
\def\ord{\mbox{\rm ord}}

\def\htt{\mbox{\rm ht}}

\def\DD{{\mathcal D}}
\def\EE{{\mathcal E}}
\def\FF{{\mathcal F}}
\def\GG{{\mathcal G}}
\def\HH{{\mathcal H}}
\def\II{{\mathcal I}}
\def\LL{{\mathcal L}}
\def\MM{{\mathcal M}}

\def\k{\mathbb{k}}




\title{Bisector fields of quadrilaterals}

 \author{Bruce Olberding} 
\address{Department of Mathematical Sciences, New Mexico State University, Las Cruces, NM 88003-8001}

\email{bruce@nmsu.edu}

\author{Elaine A.~Walker}
\address{Las Cruces, NM 88011}

\email{miselaineeous@yahoo.com}

\begin{abstract}    Working over a field of characteristic other than $2$, we examine a   relationship between quadrilaterals and  the pencil of conics passing through their vertices.   
Asymptotically, such a pencil of conics
is what we call a bisector field, a set~${\mathbb{B}}$ of paired lines such that each line $\ell$ in~${\mathbb{B}}$  simultaneously bisects each pair in~${\mathbb{B}}$ in the sense that  $\ell$  crosses the pairs of lines in ${\mathbb{B}}$  in pairs of points that all share the same midpoint. We show that a  quadrilateral induces a geometry on the affine plane via an inner product, under which we   
examine pencils of conics and pairs of bisectors of a quadrilateral. We  
show also how bisectors give a new interpretation of  some classically studied features of quadrangles, such as the nine-point conic.
%
%
%

 \end{abstract}


\thanks{\today}

\maketitle

\section{Introduction}

Throughout the paper $\k$ denotes a field of characteristic other than $2$.
Thus all the results in this article hold not only in the Euclidean setting in which $\k$ is the field of real numbers, but also, for example,  in Galois geometries where $\k$ is a finite field.
It has been known at least since the work of Steiner that if $\k$ is the field of real numbers, then the centers of the conics in the pencil of conics through the vertices of a quadrilateral (or quadrangle)
 trace out  another conic in the plane, one not in the pencil. This conic, the
    {\it nine-point conic}\footnote{In  \cite{Vac}, Vaccaro has traced the interesting and  somewhat complicated  history of the nine-point conic.} of the quadrilateral,   passes through nine canonical points of the quadrilateral; 
    see Figure~1.
     \begin{figure}[h] 
     \label{9noobconics}
 \begin{center}
\includegraphics[width=0.7\textwidth,scale=.09]{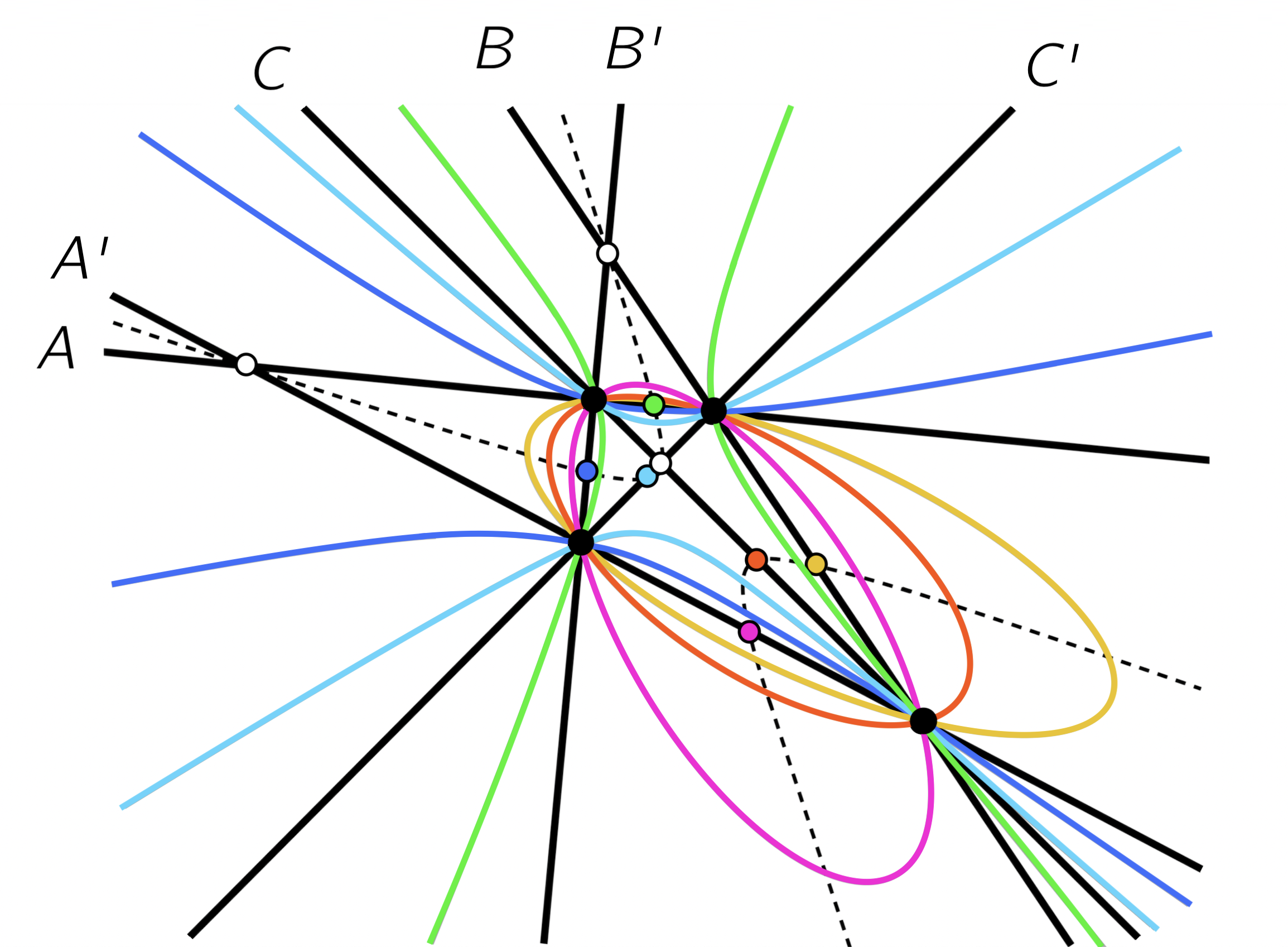} 
 \end{center}
 \caption{These nine conics (3 ellipses, 3 nondegenerate hyperbolas and 3 degenerate hyperbolas) belong to the pencil of conics through the four black points. The nine conics
 are
   centered on the midpoints and diagonal points of the quadrangle whose vertices are the four black points, and hence these centers are the nine distinguished points that define the nine-point conic for this quadrangle, which is   the dotted hyperobla.  }
\end{figure}
 In this article we explore how the geometry of the pencil of conics interacts with that of the quadrilateral and its  nine-point conic.  


 The unifying notion for our approach  is what we call a bisector of a quadrilateral, a line that crosses the pairs of opposite  sides of the quadrilateral in pairs of points that share the same midpoint; see Figures~2 and~3. 
 We prove the nine-point conic of a (proper) quadrilateral is the locus of midpoints of the bisectors of the quadrilateral. In so doing we give 
 a compact interpretation of the typically unwieldy equation that defines the nine-point conic. This is done using
 a squared norm
 that arises from an inner product  that encodes basic data of the quadrilateral.  
 
With this inner product,  a quadrilateral $Q$ induces a geometry on the affine plane, one in which  opposite sides of $Q$ are orthogonal and in which the midpoints of the bisectors are all the same ``squared distance'' from the center of the centroid. This inner product gives
 a natural way to pair the bisectors of $Q$.
 We call a pair of bisectors of $Q$ a {\it $Q$-pair} if the bisectors are orthogonal under this inner product and their midpoints have the centroid of $Q$ as a midpoint. 
 %
In particular, the pairs of opposite sides of $Q$ are $Q$-pairs of bisectors. Using results from \cite{OW4}, we show in Section 6 that there is a surprising  symmetry in this notion of bisection: Every $Q$-pair of bisectors of a quadrilateral $Q$ is bisected by every bisector of $Q$.   Thus not only does a bisector bisect the quadrilateral, the   pair to which it belongs is in return  bisected by all the other bisectors of the quadrilateral. This is expressed using what we call a {bisector field},   a collection of paired lines such that each line in the collection simultaneously  bisects every pair in the collection; see Figure~2.
   \begin{figure}[h] 
     \label{9noobconics}
 \begin{center}
\includegraphics[width=.75\textwidth,scale=.09]{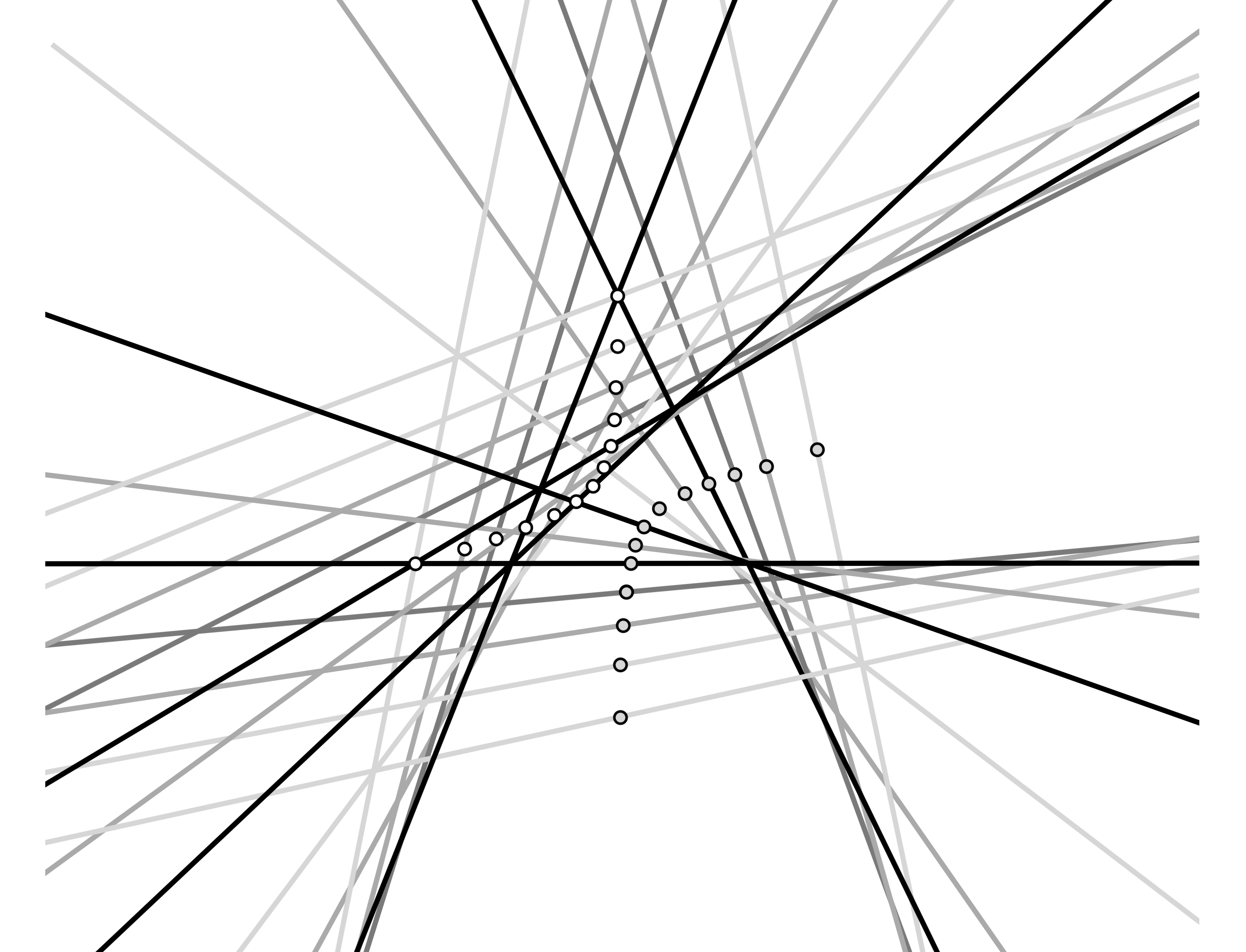} 
 \end{center}
 \caption{Pairs from a bisector field of a complete quadrilateral (whose sides and diagonals are the six black lines). 
The marked points trace out a hyperbola and  are the midpoints of the bisectors. Each bisector bisects all other pairs. The pairs are identifiable by the fact that the midpoints of the lines in the pair are antipodal on the bisector locus (the hyperbola). The points that are the intersections of the lines in the pairs lie on the branch with white points.
     }
\end{figure}
Although the original quadrilateral is present in the bisector field in the form of two pairs of lines, the bisector field is an entity that can be viewed independently  of the original quadrilateral. 

We use the symmetry of the bisection relation and results from \cite{OW4}  to interpret the pencil of conics through the vertices of a quadrilateral. We show that these conics  degenerate to the bisector field of the quadrilateral. That is,  
the degeneration of a conic in the pencil (which for a hyperbola in the pencil is simply its asymptotes) is a pair of lines that is a $Q$-pair of bisectors for the quadrilateral $Q$ of the pencil. 
Conversely, every $Q$-pair of bisectors arises as a degeneration of a conic in the pencil.  

Bisector fields are the subject of \cite{OW4}, and will be studied further in \cite{OW6}.  Except for results in Section~6, the present article is  independent from \cite{OW4}. 
Although it has no bearing on the methods or results of the present paper, we mention   that  
the present paper grew out of our interest in  planar arrangement theorems involving rectangles inscribed in quadrilaterals; see for example \cite{OW,OW2,OW3,Sch,Sch2,Tup}. In the cited articles, various aspects of the   flow of inscribed rectangles through complete quadrilaterals are studied. Affine transformations of these objects result in flows of inscribed parallelograms through complete quadrilaterals. 
We were led to the notion of a bisector in considering the extremal notion of    degenerate parallelograms inscribed in complete quadrilaterals. When such  parallelograms are extended to lines they are precisely the bisectors of the quadrilaterals. These bisectors serve as the unifying idea for this article and \cite{OW4} and \cite{OW6}. 


%



\smallskip

{\bf Notation and terminology}. 
By a {\it $($complete$)$ quadrilateral} $Q = ABA'B'$ we mean a arrangement of four distinct lines $A,B,A',B'$,  called the {\it sides of $Q$},  and the six points of intersections of these four lines. We require that not all four lines go through a single point. Also,  
adjacent sides of $Q$  are not allowed to be parallel but opposite sides are. (Here we are being more restrictive in the definition of quadrilateral than in \cite{OW4}, where adjacent sides are allowed to be parallel.) 
The intersections of adjacent sides are the {\it vertices} of $Q$. 
The two lines through nonadjacent vertices are the {\it diagonals} of $Q$.  
Because of some technical considerations later in the paper and in \cite{OW6}, we  deviate from conventional terminology and allow  a quadrilateral to have three sides that pass through a single point, and so two vertices of $Q$ (but at most two) can coincide. 
 In this case we say the quadrilateral $Q$ is {\it improper}, and otherwise $Q$ is {\it proper}.    The {\it centroid} of a quadrilateral $Q$ is the midpoint of the midpoints of the diagonals of $Q$.  Equivalently, the {centroid} of $Q$ is the midpoint of the midpoints of any pair of opposite sides of $Q$.

If a quadrilateral is proper, its four distinct vertices define a quadrangle in the traditional sense:
A {\it quadrangle} $\Q$ consists of four distinct points in the plane, the {\it vertices} of $\Q$, and the six lines through them, the {\it sides} of $\Q$. A quadrilateral $Q$ {\it belongs} to a quadrangle $\Q$ if  the vertices of $Q$ are the same as those of $\Q$ and hence the four sides of the quadrilateral~$Q$ are among the six lines through the vertices of $\Q$.  
There are three quadrilaterals that belong to a quadrangle, and each of these  is necessarily proper. The improper quadrilaterals are precisely the quadrilaterals that do not belong to a quadrangle. By a {\it parallelogram}  we mean a quadrilateral whose pairs of opposite sides are parallel. 
 The centroid of a quadrangle $\Q$  is the (shared) centroid of any quadrilateral belonging to $\Q$. 

We occasionally need to work in projective space. We denote the projective closure of  $n$-dimensional affine space ${\mathbb{A}}^n(\k)=\k^n$ by $\P^n(\k)$. The points in $\P^n(\k)$ are written in homogeneous coordinates $[x_1:x_2: \cdots :x_{n+1}]$, where $x_1,x_2,\ldots,x_{n+1} \in \k$ are not all zero and 
 this point is the   line   in $\k^{n+1}$    that passes through the origin and the point $(x_1,x_2,\ldots,x_{n+1})$. We identify ${\mathbb{A}}^n(\k)$ with the set of points in $\P^n(\k)$ of the form $[x_1:\cdots :x_{n}:1]$.  
We will only need $\P^1(\k)$ and $\P^2(\k)$ in this article. By  {\it the line at infinity} in $\P^2(\k)$ we mean the line $\{[x_1:x_2:0] : x_1,x_2 \in \k$ not both zero$\}$.






 Throughout the paper, we use the following convention. 
 \begin{quote}
{\it If $tX- uY+v=0$ defines a line $L$ in the plane  $\k \times \k$, then  we can assume without loss of generality that if $u =0$ then $t=1$, and if $u \ne 0$ then $u =1$. 
With this assumption, we have a canonical choice of coefficients for the line, and so we may unambiguously  write $t_L $ for $t$, $u_L$ for $u$ and $v_L$ for $v$. 
}  \end{quote}

\section{$Q$-orthogonal pairs }


For a quadrilateral $Q$, we introduce  an inner product on $\k \times \k$ under which the pairs of opposite sides and diagonals of $Q$ are orthogonal. In the next definition we are using the notational convention from the last section. 

\begin{definition} \label{alpha def notation}
 Let $Q=ABA'B'$ be a quadrilateral, possibly improper, and let  \begin{eqnarray*}
\alpha & = & t_Au_Bu_{A'}u_{B'} - u_At_Bu_{A'}u_{B'} + u_Au_Bt_{A'}u_{B'}  - u_Au_Bu_{A'}t_{B'} \\
%
\beta 
& = & 
 t_Au_Bt_{A'}u_{B'}-u_At_Bu_{A'}t_{B'}\\
\gamma 
& = &  t_At_Bt_{A'}u_{B'}-t_At_Bu_{A'}t_{B'}+t_Au_Bt_{A'}t_{B'}-u_At_Bt_{A'}t_{B'}.
  \end{eqnarray*}  
Let $\Phi_Q$ be the quadratic form $$\Phi_Q(X,Y) = \gamma X^2 -2\beta XY +\alpha Y^2.$$ 
  For the vector space $\k^2$,  define an inner product (i.e., symmetric bilinear form) 
         by 
  $$\langle {\bf v},{\bf w} \rangle_Q = {\bf v}^T
  \begin{bmatrix} 
  \gamma & -\beta \\
  -\beta & \alpha
  \end{bmatrix}
  {\bf w} \:  {\mbox{ for all }} {\bf v},{\bf w} \in \k^2. 
  $$
  Therefore,  $\Phi_Q({\bf v}) = \langle {\bf v},{\bf v} \rangle_Q$ for all $ {\bf v},{\bf w} \in \k^2.$

%
  \end{definition}

The quadratic form $\Phi_Q$ can be viewed as a squared norm on $\k^2$. 
Different quadrilaterals can  induce the same inner product and squared norm since these objects depend only on the slopes of the sides of the quadrilateral, not their position.

  \begin{lemma} \label{nondeg form} The inner product $\langle -,-\rangle_Q$ is nondegenerate. 
  \end{lemma} 
  
  \begin{proof} 
  It is straightforward to verify
\begin{eqnarray} \label{disc eq} 
\beta^2-\alpha \gamma & = &  (t_Au_B-t_Bu_A)(t_Bu_{A'}-t_{A'}u_B)(t_{A'}u_{B'}-t_{B'}u_{A'})(t_{B'}u_A-t_Au_{B'}).
\end{eqnarray}
 Since adjacent sides of $Q$ are not parallel, $\beta^2 - \alpha\gamma \ne 0$. The determinant of the matrix in Definition~\ref{alpha def notation}   is $\alpha\gamma-\beta^2$, so this matrix is invertible, and hence the inner product is nondegenerate. 
  \end{proof}
  

\begin{definition}
Two lines $\ell_1,\ell_2$ in $\k^2$ are {\it $Q$-orthogonal} if
$\langle (u_{\ell_1},t_{\ell_1}),(u_{\ell_2},t_{\ell_2})\rangle_Q=0.$
  \end{definition}


To prove the main result of this section, Proposition~\ref{alpha def}, we recall that
an {\it involution} (or more precisely, an involutive homography) on a line $\ell$  in the projective plane $\P^2(\k)$ is a projective transformation $\Lambda$ from $\ell$ to $\ell$ for which $\Lambda \circ \Lambda$ is the identity map on $\ell$. 
Two points $p$ and $q$ on $\ell$ are {\it conjugate} under $\Lambda$ if $\Lambda(p) = q$. 

\begin{lemma} 
\label{des}  {\rm (Desargues' Involution Theorem)} 
Let $\Q$ be a quadrangle in $\P^2(\k)$, and let $\ell$ be a line in $\P^2(\k)$   that  does not pass through a vertex of $\Q$. 
 There is an involution $\Lambda$ on $\ell$ such that if $p$ and $q$ are the points of intersection of $\ell$ and a conic through the vertices of $\Q$, then $p$ and $q$ are conjugate under $\Lambda$. Moreover, any two conjugate pairs of points on $\ell$ determine the involution. 
\end{lemma} 

For a proof of Lemma~\ref{des}
 in our setting of arbitrary fields of characteristic other than $2$, see \cite[14.2.8.3, p.~125]{Berger} or \cite{Ngy}. For a proof using bisector methods, see \cite[Corollary 5.6]{OW4}.


 \begin{proposition}  \label{alpha def} 
\label{des2} 
If $Q$ is a quadrilateral, possibly improper, then the 
 pairs of opposite sides of $Q$ are $Q$-orthogonal, as are the diagonals of $Q$.  
 \end{proposition}


\begin{proof}  Write $Q= ABA'B'$, and 
 let $\Lambda_Q$ denote the  projective transformation   on
 the line at infinity 
   given for each point $[x:y:0]$  by
 $$\Lambda_{Q}([x:y:0]) = [\beta x-\alpha y  :\gamma x - \beta y :0].$$  
  Since $\beta^2-\alpha \gamma \ne 0$ by Lemma~\ref{nondeg form}, $\Lambda_Q$ is an involution on the line at infinity.      A straightforward calculation using the expressions for $\alpha,\beta,\gamma$ in Definition~\ref{alpha def notation} shows 
     the points at infinity $[u_A:t_A:0]$ and $[u_{A'}:t_{A'}:0]$
 for $A$ and $A'$ are conjugate under $\Lambda_Q$, as are the points at infinity for $B$ and $B'$. 
 Comparing the definition of $\Lambda_Q$ with that of the inner product in Definition~\ref{alpha def notation}, it follows that 
 two lines in  $\k^2$ are $Q$-orthogonal if and only if their points at infinity   are conjugate under $\Lambda_Q$.  

%
It remains to consider the diagonals of $Q$. 
 If $Q$ is an improper quadrilateral, the diagonals of $Q$ are opposite sides of $Q$ and so the diagonals are $Q$-orthogonal. Assume $Q$ is a proper quadrilateral. 
Let $\Q$ be the quadrangle to which $Q$ belongs. 
By Lemma~\ref{des} the quadrangle $\Q$   defines an involution $\Lambda$ on the line at infinity  such that the points at infinity of any conic through the vertices of $\Q$ are conjugate, and this involution is uniquely determined by 
any two conjugate pairs. 
 Each pair of  opposite sides  of $\Q$ comprises a (degenerate) conic   passing through the vertices of $\Q$, and so the points at infinity for opposite sides of $\Q$ are conjugate under $\Lambda$. Therefore, since the points at infinity for $A$ and $A'$  are conjugate under $\Lambda$ and $\Lambda_Q$, as are the points at infinity for $B$ and $B'$, we have
 $\Lambda = \Lambda_Q$ by Lemma~\ref{des}.   
Consequently,  the diagonals of $Q$, as opposite sides of $\Q$, have points at infinity that are conjugate under $\Lambda_Q$. Hence the diagonals of $Q$ are $Q$-orthogonal. 
\end{proof}

 
 %


 
  
  \begin{remark} \label{self} With $\Lambda_Q$ as in the proof of Proposition~\ref{alpha def}, 
the points $[x:y:0]$ on the line at infinity that are conjugate to themselves correspond to  the  null vectors for $\Phi_Q$, that is,  $$\Lambda_Q([x:y:0]) = [x:y:0] \: \Longleftrightarrow \: \Phi_Q(x,y) =0.$$ Thus a line   is $Q$-orthogonal to itself if and only if its point at infinity is conjugate to itself under $\Lambda_Q$.   
  \end{remark}
 
 \begin{corollary}  \label{where needed}
If  $Q$ and $Q'$ are proper quadrilaterals sharing the same vertices, then two lines are $Q$-orthogonal if and only if they are $Q'$-orthogonal.
\end{corollary}

\begin{proof} 
The proof of Proposition~\ref{alpha def} shows that  $\Lambda_Q = \Lambda_{Q'}$ since the points at infinity for the pairs of opposite sides and diagonals of $Q$ are conjugate under both $\Lambda_Q$ and $\Lambda_{Q'}$.  Also, as noted in the proof of Proposition~\ref{alpha def}, two lines are $Q$-orthogonal if and only if their points at infinity are conjugate under $\Lambda_Q$. The analogous statement holds for $\Lambda_{Q'}$, and so the corollary follows.
\end{proof} 

The next proposition will be needed later when using affine transformations to reduce to simpler cases. 

\begin{proposition} \label{image orthogonal} Let $Q$ be a quadrilateral, and let $f$ be an invertible linear transformation. There is $\lambda \in \k$ such that for all ${\bf v},{\bf w} \in \k^2$,  $$\left<{\bf v},{\bf w}\right>_Q =\lambda \left< f({\bf v}),f({\bf w})\right>_{f(Q)}.$$  Thus a pair of lines $\ell_1,\ell_2$ is $Q$-orthogonal if and only if $f(\ell_1),f(\ell_2)$ is $f(Q)$-orthogonal. 
\end{proposition}

\begin{proof} Write $Q = ABA'B'$, and let $a,b,c,d \in \k$ such that $f(x,y) = (ax+by,cx+dy)$ for all $x,y \in \k$.  Define a map $\f_\infty$ on the line at infinity by 
\begin{center} $f_\infty([x:y:0]) = [ax+by:cx+dy:0]$ for all points $[x:y:0]$ on the line at infinity.
\end{center} 
Let $\Lambda_Q$ be the involution   from the proof of Proposition~\ref{alpha def}. Then $f_\infty \Lambda_Q f_\infty^{-1}$ is an involution on the line at infinity for which 
\begin{center} $[u_{f(A)}:t_{f(A)}:0], 
[u_{f(A')}:t_{f(A')}:0]$ 
 and $[u_{f(B)}:t_{f(B)}:0], 
[u_{f(B')}:t_{f(B')}:0]$ 
\end{center}
are conjugate pairs. Since these pairs are pairs of points at infinity of sides of the quadrilateral $f(Q)$, they are conjugate  under $\Lambda_{f(Q)}$,  and so Lemma~\ref{des} implies $\Lambda_{f(Q)} = f_\infty \Lambda_Q f_\infty^{-1}$.  Let $x,y$ be in the algebraic closure ${\overline{\k}}$ of $\k$, where at least one of $x,y$ is nonzero. 
Using Remark~\ref{self}, $\Phi_{f(Q)}(f(x,y)) = 0$ if and only if  $\Phi_Q(x,y)=0$.
 Thus $\Phi_{f(Q)} \circ f$ and $\Phi_Q$ have the same zeroes in $\P^1(\overline{\k})$, and so 
  there is $\lambda$ in the algebraic closure of $\k$  such that $\lambda \Phi_Q =  \Phi_{f(Q)} \circ f$.
Moreover,   $\lambda \in \k$    
since the coefficients of these two polynomials are in $\k$.   The proposition now follows from the parallelogram law for inner products. 
\end{proof}


\section{Basic properties of bisectors}

Bisectors of collections of conics are defined in \cite{OW4}, and  we may apply this definition to our setting. 
A line $\ell$ {\it crosses  a pair of lines} $\PP = \{\ell_1,\ell_2\}$ if it is distinct from both $\ell_1$ and $\ell_2$ and is not   parallel to both of them.  In this case, we write $\mid_{\PP}(\ell)$ for the midpoint of the two points where $\ell$ meets $\PP$. At most one of these two points may be at infinity, and if one is at infinity we define $\mid_{\PP}(\ell)$ to be the point at infinity for $\ell$. If the midpoint is not at infinity, we say it is finite. 
Figure~3 illustrates the next definition.

\begin{definition} \label{bisector def}
Let $\ell$ be a line. 
\begin{itemize}
%
\item[(1)] If ${\mathcal{P}}$ is a collection of pairs of lines, then 
 $\ell$ {\it bisects  ${\mathcal{P}}$}    if $\mid_{\PP}(\ell) = \mid_{\PP'}(\ell)$ for all  pairs $\PP,\PP' \in {\mathcal{P}}$ that $\ell$ crosses. This common midpoint  is the {\it midpoint} of the bisector~$\ell$.
 
 %
 %

\item[(2)] If $Q$ is a quadrilateral or quadrangle, then 
$\ell$ {\it bisects   $Q$} if $\ell$ bisects the 
 pairs of opposite sides of $Q$. We write $\mid_Q(\ell)$ for the midpoint of the bisector $\ell$. 

\end{itemize}

\end{definition}

    \begin{figure}[h] 
     \label{9noobconics}
 \begin{center}
\includegraphics[width=.95\textwidth,scale=.09]{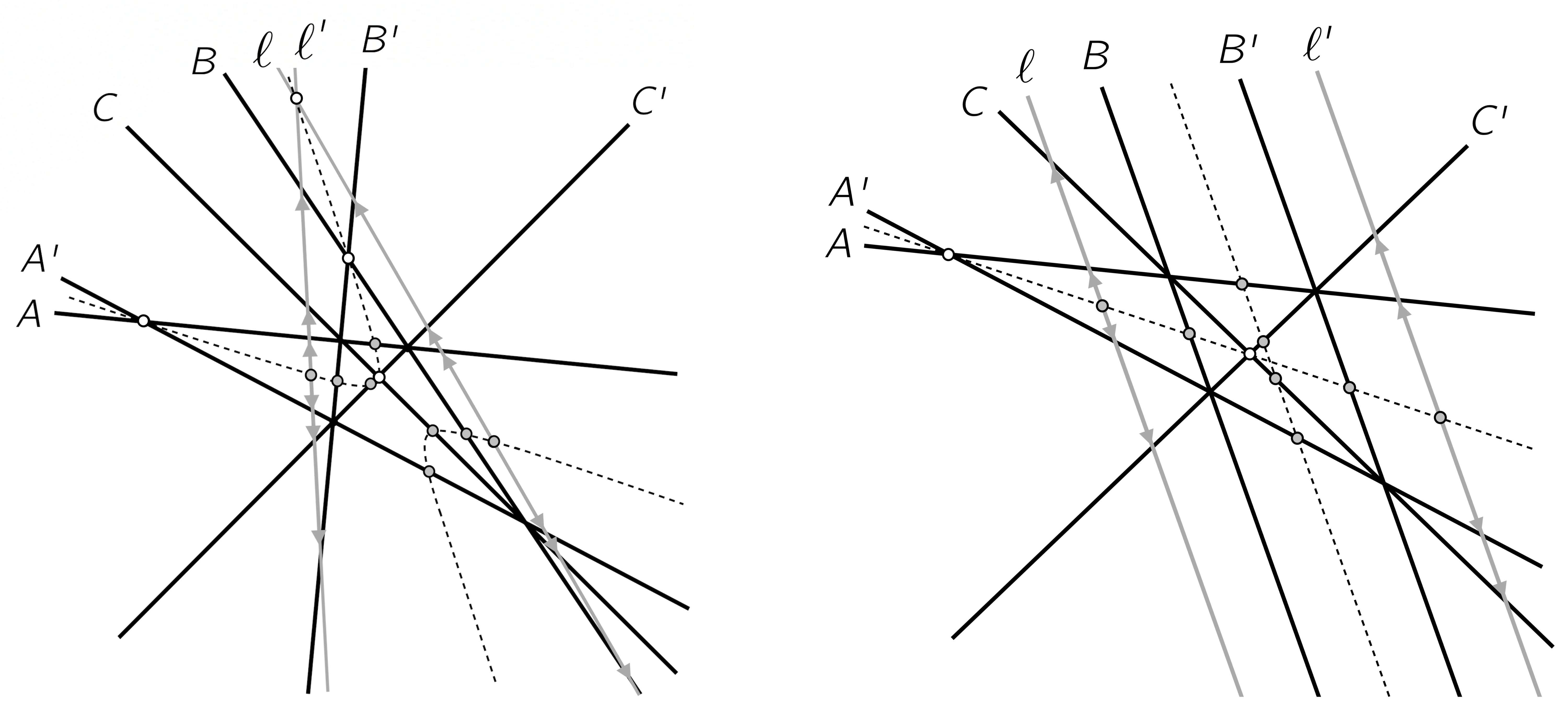} 
 \end{center}
 \caption{In each of these two arrangements,  
 the lines $\ell$ and $\ell'$ bisect the quadrangle defined by the three pairs of lines $A,A'$; $B,B'$; and $C,C'$. The arrows on $\ell$ and $\ell'$ help visualize the bisection property. The midpoints of the bisectors lie on the nine-point conic of the quadrangle and are marked in gray; see Section~5. That the figure at right has a degenerate nine-point conic is because two sides of the quadrangle are parallel; see Theorem~\ref{deg}.  }
\end{figure}

When referring to a bisector $\ell$ of a quadrilateral $Q$, we mean the line and its midpoint $\mid_Q(\ell)$.   
This point is determined by 
 any two points where the bisector crosses a pair of opposite sides of $Q$ that are not parallel to~$\ell$.  (There must be at least one such pair of opposite  sides since the vertices of $Q$ are, by definition, not at infinity.)



\begin{proposition} \label{sides} \label{parallel sides 0}
Let $Q$ be a  quadrilateral.
  \begin{enumerate}
  \item[$(1)$] The midpoint $\mid_Q(\ell)$ is   finite   for all bisectors $\ell$ of $Q$.

  \item[$(2)$]
 A line $\ell$  passing through a vertex of $Q$  bisects $Q$ if and only if it is a side of $Q$ or a diagonal of $Q$.    

\item[$(3)$] Two distinct parallel lines are bisectors of $Q$ if and only if these lines are parallel to a pair of   sides or diagonals of $Q$.   

\item[$(4)$] All lines parallel to a pair of  parallel sides or diagonals of $Q$ bisect $Q$.  

\end{enumerate}

\end{proposition}

\begin{proof} Statement (1) follows from the fact that adjacent sides of a quadrilateral (as defined in the introduction) cannot be parallel. Statement (2) is a simple argument that depends on the fact that any bisector that goes through one vertex of a quadrilateral must go through another. For (3) and (4), use the fact that if $\PP$ is a pair of lines, and $\ell$ and $\ell'$ are distinct parallel lines not parallel to either line in $\PP$, then for every line $\ell''$ parallel to $\ell$ and $\ell'$,  the line through $\mid_{\PP}(\ell)$ and $\mid_{\PP}(\ell')$ goes through $\mid_{\PP}(\ell'')$. This fact implies that if $\ell$ and $\ell'$ are bisectors of $Q$, then every line  
 parallel to $\ell$ and $\ell'$ is a bisector of $Q$. Thus there are bisectors through the vertices of $Q$ parallel to $\ell$ and $\ell'$, and by (2), these bisectors must be sides or diagonals of $Q$.  
Statements (3) and (4) follow from these considerations. 
\end{proof}


See Figure 3 for an illustration of (3) and (4) of the lemma.

 \begin{lemma} \label{sides quad} 
  Let $\Q$ be a quadrangle.
The bisectors of $\Q$ that pass through a vertex of $\Q$ are the sides of $\Q$. 
 The bisectors  that do not pass through a vertex  are the lines whose involutions induced by $\Q$ as in Lemma~\ref{des} are reflections. 
\end{lemma}

\begin{proof} 
The first assertion follows from Proposition~\ref{sides}(2).
To prove the second, 
suppose a line $\ell$ bisects $\Q$ but does not pass through a vertex of $\Q$. The involution  on $\ell$ given by Lemma~\ref{des}  is determined by where $\ell$ meets pairs of opposite sides of $\Q$ (which is possibly at infinity), and so  
$\ell$ is a bisector of $\Q$ if and only if 
    this involution is a reflection. 
\end{proof}

By definition, bisection  of quadrangles requires bisection of all three pairs of opposite sides. Proposition~\ref{two three} shows it is enough to check bisection for two pairs of opposite sides.

 \begin{proposition} \label{two three}  
If $Q$ and $Q'$ are proper quadrilaterals sharing the same vertices,
 then  the bisectors of $Q$ and  $Q'$ are the same. Also, 
    $\Phi_Q = \lambda \Phi_{Q'}$ for some $\lambda \in \k$,
    and two lines are $Q$-orthogonal if and only if they are   $Q'$-orthogonal.
 
\end{proposition}

\begin{proof}   
By assumption,   $Q$ and $Q'$ belong to the same quadrangle $\Q$.   
To prove the first assertion, it suffices to  show that if  $\ell$ is a line that bisects a quadrilateral $Q$, then $\ell$ bisects $\Q$. Let $\ell$ be such a line. 
If $\ell$  passes through a vertex of $\Q$, then 
 $\ell$ is a side  of $\Q$   by Proposition~\ref{sides}, and hence $\ell$ is a bisector of $\Q$. If instead $\ell$ does not pass through a vertex of $\Q$, then by Lemma~\ref{des} the involution on $\ell$ induced by
 $\Q$
  is determined by the points where  two pairs of opposite sides of $\Q$   meet $\ell$. Since $\ell$ bisects these pairs and an involution on a line is determined by any two pairs of conjugate points,  this involution   is a reflection and so $\ell$ bisects $\Q$ by Lemma~\ref{sides quad}.

 To prove the rest of the proposition,  note that 
 by equation~(\ref{disc eq}) from the proof of Lemma~\ref{nondeg form}, the 
discriminant of $\Phi_Q$ is nonzero, so 
  the quadratic form $\Phi_Q$ has  two distinct zeroes in $\P^1(\overline{\k})$, where $\overline{\k}$ is the algebraic closure of $\k$.  Similarly, $\Phi_{Q'}$ has two distinct zeroes. 
  By  Corollary~\ref{where needed}, two lines are $Q$-orthogonal if and only if they are $Q'$-orthogonal, so it follows from Remark~\ref{self} that $\Phi_{Q}$ and $\Phi_{Q'}$ have the same zeroes in $\P^1(\overline{\k})$. 
    Therefore, $\Phi_Q = \lambda \Phi_{Q'}$ for some $\lambda \in \overline{\k}$. In fact, $\lambda \in \k$  since the coefficients of $\Phi_Q$ and $\Phi_{Q'}$ are in $\k$. 
\end{proof}




The next proposition, which shows that two bisectors can share the same midpoint only in a very special case, will be used often in what follows. 

\begin{proposition}   \label{unique}
Let $Q$ be a quadrilateral. Two distinct bisectors of $Q$  share the same midpoint if and only if
the  vertices of $Q$ are  the vertices of a parallelogram.
 In this case, the shared midpoint is  the centroid of~$Q$. 
\end{proposition}

\begin{proof}  
We will use the following observation, which can be checked with an easy calculation:
\begin{itemize}
\item[$(\dagger)$] If $\PP$ is a pair of  non-parallel lines  and $p$ is a point in the plane that is not the intersection of the lines in $\PP$, then  there is a unique line $\ell$ passing through $p$ such that $p = \mid_\PP(\ell)$.
\end{itemize}  
Write $Q = ABA'B'$, and suppose first that 
  $Q$ is an improper quadrilateral. Then the vertices of $Q$ are not the vertices of a parallelogram. Also, three sides, say $A,A',B$, meet at a point.  
    If $p$ is a point not at this intersection, then by ($\dagger$) there is at most one bisector whose midpoint is $p$. Otherwise, suppose $p$ is  the point of intersection of $A,A'$ and $B$. 
    %
  If there are two bisectors with midpoint $p$, then
  these two bisectors are sides or diagonals  of $Q$
    by Proposition~\ref{sides}(2) and hence must be among $A,A'$ and $B$. Thus at least one of $A$ or $A'$ has  midpoint $p$, say $A$,  and so  
     $A$ bisects $B,B'$ with midpoint $p$. Since $B$ passes through $p$, this implies $B'$ also passes through $p$,  a contradiction to the fact that not all four sides of $Q$ go through a point. Thus if $Q$ is an improper quadrilateral, the vertices of $Q$ are not the vertices of a parallelogram and no two bisectors share a midpoint.

Now suppose $Q$ is a proper quadrilateral.
 If the vertices of $Q$ are not the vertices of a parallelogram, then there is a quadrilateral $Q'$  that has the same vertices as  $Q$ such that no sides of $Q'$ are parallel.  Thus if $p$ is a point in the plane, there is a pair of non-parallel opposite sides of $Q'$ such that $p$ is not the intersection of these two lines. and so by ($\dagger$), there is at most one bisector having  $p$ as its midpoint.  
Otherwise, if the vertices of $Q$ are the vertices of a parallelogram $P$, then any line through the center $p$ of $P$ is a bisector of $P$ having midpoint $p$. Since $Q$ and $P$ belong to the same quadrangle, Proposition~\ref{two three} implies these lines through $p$ are bisectors of $Q$ having midpoint $p$. 

Finally, if $\ell$ and $\ell'$ are different bisectors of $Q$ with the same midpoint $p$, then, as we have shown, the vertices of $Q$ are the vertices of a parallelogram $P$ and the centroid of $Q$ is the center of $P$.
By Proposition~\ref{two three}, $\ell$ and $\ell'$ bisect the diagonals of the parallelogram $P$, and so by ($\dagger$), since $\ell$ and $\ell'$ are distinct, $p$ must lie on the intersection of the diagonals, and hence $p$ is the center of $P$, which is the centroid of~$Q$.
%
\end{proof}

 Proposition~\ref{unique} is not true for the more general quadrilaterals  considered in \cite{OW4}, those quadrilaterals that can have a vertex at infinity. See   \cite[Figure 2(c)]{OW4}.






%






\section{Quadrilaterals in standard form}

For use in later sections, 
we given an equational description of all the  bisectors of a quadrilateral $Q$ in what we call standard form, a simpler case we can always reduce to via an affine transformation if $Q$ is not a parallelogram.    

 
 \begin{definition} \label{standard def}
A quadrilateral $Q =ABA'B'$ is in {\it standard form} if $A$ is the line $Y=0$ and  $A'$ is the line $X=0$. 
\end{definition}

If $Q$ is a quadrilateral that is not a parallelogram, then $Q$ has a pair of opposite sides, say $A,A'$, that are not parallel, and so there is an affine transformation that carries $A$ and $A'$ onto the axes of $\k^2$.    Thus every quadrilateral that is not a parallelogram can be transformed into a quadrilateral in standard form. We use this often in the next two sections. 


If the quadrilateral $Q=ABA'B'$ is in standard form, then neither of the lines $B$ or $B'$ is parallel to the $X$ or $Y$-axis, and so $u_B=u_{B'} =1$ (with our notational convention from the Introduction), which implies the slopes of the lines $B$ and $B'$ are $t_B$ and $t_{B'}$, respectively. The product  $t_Bt_{B'}$ of these slopes is a fundamental invariant of $Q$. 

\begin{definition} The {\it coefficient} of a quadrilateral $Q = ABA'B'$ in standard form is the product $\mu = t_Bt_{B'}$ of the slopes of $B$ and $B'$.  
\end{definition} 

The coefficient   is never $0$ since the slopes of $B$ and $B'$ are not $0$. 
In standard form, $Q$-orthogonality is easy to detect in terms of the coefficient of $Q$:


 \begin{lemma}  \label{standard phi} 
If a quadrilateral $Q$ is in standard form with coefficient $\mu$, 
 then  $$\Phi_Q(X,Y) = Y^2 - \mu X^2.$$  
Moreover, two lines $\ell$ and $\ell'$ not parallel to the $X$ or $Y$-axes are $Q$-orthogonal if and only if the product of their slopes is    $\mu$.  

\end{lemma} 
 


\begin{proof} 
Write $Q=ABA'B'$, where the sides of $Q$  are as in Definition~\ref{standard def}.  Substituting the data for $A$, $B$, $A'$ and $B'$ into the definitions
of $\alpha,$ $\beta$ and $\gamma$ in Definition~\ref{alpha def notation}, we obtain $\alpha =1, \beta=0, \gamma =-t_Bt_{B'}=-\mu$, and so $\Phi_Q(X,Y) = \gamma X^2 - 2\beta XY +\alpha Y^2 =  Y^2-\mu X^2.$
Also, 
 since $u_{\ell}=u_{\ell'}=1$, the lines $\ell$ and $\ell'$ are $Q$-orthogonal if and only if $t_\ell t_{\ell'} -\mu =0$.  \end{proof} 

The next proposition  gives an equational description of all of the bisectors of a quadrilateral in standard form.  

\begin{proposition} \label{brand new1}  \label{brand new2} Let $Q$ be a quadrilateral in standard form with   centroid $(h,k)$ and coefficient~$\mu$.
\begin{itemize}
\item[$(1)$] A line $\ell$
is a bisector of $Q$ with midpoint $(p,q) \ne (0,0)$ if and only if 
\begin{center}$q(q-2k)-\mu p(p-2h)=0$ and $\ell$ has equation $qX+pY-2pq=0$. \end{center} 

\item[$(2)$]  If  $(h,k) = (0,0)$, every line through $(0,0)$ is a bisector of $Q$ with   midpoint $(0,0)$. 

  \item[$(3)$] If   $(h,k) \ne (0,0)$,  then the unique bisector whose midpoint is $(0,0)$ is $kY+\mu hX=~0$.


\end{itemize}

\end{proposition}


\begin{proof} (1) Write $Q = ABA'B'$. 
Since $A$ is the line $Y=0$ and $A'$ the line $X=0$, 
 the unique line
$\ell$ passing through $(p,q)$ that intersects $A$ and $A'$ at points that have $(p,q)$ as their midpoint is the line passing through $(0,2q)$ and $(2p,0)$. This line $\ell$ has equation $qX+pY-2pq=0$, and if $\ell$ is distinct from $A$ and $A'$, the line $\ell$ is the unique bisector of $A,A'$ with midpoint $(p,q)$.   
%
%
%
Thus  it suffices to show:
 \begin{itemize}
 \item[($\star$)] {\it The  line $qX+pY-2pq=0$   bisects $Q$ with midpoint $(p,q) $ if and only if $q(q-2k)-\mu p(p-2h)=0$.}  \end{itemize}

We first calculate the centroid of $Q$. Since $Q$ is in standard  form, the equations for $B$ and $B'$ are given by $Y=t_BX+v_B$ and $Y = t_{B'}X+v_{B'}$, with $t_Bt_{B'} =\mu$. 
The vertices of the quadrilateral $Q$ are 
 \begin{eqnarray} \label{vertices}
 A \cdot B = \left(-\frac{t_{B'}v_B}{\mu},0\right), \: B \cdot {A'} = (0,v_B), \: {A'} \cdot {B'} = (0,v_{B'}), \: {B'} \cdot A = \left(-\frac{t_Bv_{B'}}{\mu},0\right).
 \end{eqnarray}  Therefore,  since $(h,k)$ is  the mean of the four vertices of $Q$, 
 \begin{eqnarray} \label{centroid eq}
 -4h\mu = t_{B'}v_B + t_Bv_{B'} \: {\mbox{ and  }} \: 4k = v_B+v_{B'}.
 \end{eqnarray}


Now we verify $(\star)$. 
Suppose first  $\ell$ is not parallel  to $B$ or $B'$. 
 Then $pt_B + q\ne 0$ and $pt_{B'} + q\ne 0$, and so the intersections of $\ell$ with $B$ and $B'$ are the points 
 $$\ell \cdot B = \left(\frac{p(2q-v_B)}{pt_B+q},\frac{q(2pt_B+v_B)}{pt_B+q}\right) \:\: \:\: 
 \ell \cdot B' = \left(\frac{p(2q-v_{B'})}{pt_{B'}+q},\frac{q(2pt_{B'}+v_{B'})}{pt_{B'}+q}\right)$$
The line $\ell$ bisects the pair $B,B'$ with midpoint $(p,q)$ if and only if $(p,q)$ is the midpoint of $\ell \cdot B$ and $\ell \cdot B'$. Using this observation, the expressions for $h$ and $k$ in (\ref{centroid eq}), and the fact that $t_Bt_{B'}=\mu $ and $(p,q) \ne (0,0)$, 
 a  midpoint calculation shows $\ell$ bisects the pair $B,B'$ with midpoint $(p,q)$  if and only if 
$q(q-2k)-\mu p(p-2h)=0.$ This verifies the proposition in the case in which $\ell$ is not parallel to $B$ or $B'$. 
 
Now suppose  $\ell$ is parallel to $B$ or $B'$.  
 We give the argument for the case in which $\ell$ is  parallel to $B$. The case in which $\ell$ is parallel to $B'$ is similar. 
 Since by assumption  $qX+pY-2pq$  is parallel to $B$, and 
 $B$ is not parallel to $X=0$ or $Y=0$, it follows that 
 $p \ne 0$ and $q \ne 0$. 
 The slope of $\ell$ is $-q/p$, so $q=-pt_B$.  
 Using this fact, the assumption that $t_Bt_{B'}=\mu$ and the expressions for $h$ and $k$ from (\ref{centroid eq}), we obtain
  $$q(q-2k)-\mu p(p-2h) 
  = \frac{p}{2}(t_B-t_{B'})(2pt_B+v_B).$$
 Since $p \ne 0$,   
 \begin{center}
$q(q-2k)-\mu p(p-2h) =0 \:\: \Longleftrightarrow \:\: t_B=t_{B'}$ or $2pt_B+v_B=0$.
 \end{center}
To complete the proof of ($\star$), it suffices to show (still under the assumption that $\ell$ is parallel to $B$) that the line $qX+pY-2pq=0$  bisects $Q$ with midpoint $(p,q)$ if and only if  $t_B=t_{B'}$ or $2pt_B+v_B=0$.

If $\ell = B$, then since $\ell$ goes through the point $(2p,0)$ and $B$ has equation $Y=t_BX+v_B$, it follows that $2pt_B + v_B =0$.
If instead
 $\ell \ne B$ and  $\ell$ 
bisects $B,B'$ with midpoint $(p,q)$, then since $\ell$ is  a bisector distinct from $B$ and parallel to $B$,  Proposition~\ref{sides}(3) implies $B$ is parallel to $B'$  so that $t_B=t_{B'}$.
 
 Conversely, still under the assumption that $\ell$ is parallel to $B$, suppose  $t_B=t_{B'}$ or $2pt_B+v_B=0$. In the former case, $B$ is parallel 
  to $B'$ and so since $\ell$ is parallel to $B$ and $B'$,   $\ell$ bisects $Q$ with  midpoint $(p,q)$. In the  case in which $2pt_B+v_B=0$,   the parallel lines $\ell$ and $B$ share the point $(2p,0)$ and hence are equal. Thus $\ell$ bisects $B,B'$ with  midpoint $(p,q)$. This completes the proof of ($\star$). 


(2)  Since $Q$ is in standard form and all four lines do not go through the same point, it cannot be that $v_B = v_{B'} = 0$. 
Thus if   $(h,k) =(0,0)$, then   equations  (\ref{centroid eq}) imply $v_B = -v_{B'}$ and $t_B = t_{B'}$, which in turn implies $B$ and $B'$ are parallel lines with $(0,0)$ lying on their midline\footnote{The {\it midline} of a pair of parallel lines $L$ and $L'$  is the line consisting of the midpoints of the points on $L$ and the points on $L'$, i.e., it is the line midway between the two parallel lines.}, and so every line through the origin bisects the pair $B,B'$ with the origin as its midpoint.  Since the origin is the intersection of $A$ and $A'$,   every line through the origin bisects $Q$ with the origin as its midpoint.

(3)  Suppose  $(h,k) \ne (0,0)$.  
We show the line $\ell$ defined by $kY+\mu hX=0$ bisects $Q$ with midpoint $(0,0)$. Once this is proved, the fact that $\ell$ is the unique bisector with midpoint $(0,0)$ follows  from Proposition~\ref{unique}.  Note that the   line $\ell$ bisects $A,A'$ with midpoint $(0,0)$ since $A$ and $A'$ meet at the origin. Thus it remains to show $\ell$ bisects the pair $B,B'$ with midpoint $(0,0)$. 
If $B$ is parallel to $B'$, then using the expressions for $h$ and $k$  from equations (\ref{centroid eq}), we have 
 $-4 \mu h = t_Bv_{B'}+t_{B'}v_B= 4 t_Bk$. Thus $-\mu h = t_B k$, 
 proving that 
 $\ell$ is parallel to $B$ and $B'$. 
 In this case, 
 $\ell$ bisects $Q$ with midpoint $(0,0)$ because $\ell$ bisects $A,A'
 $ with midpoint $(0,0)$ and is parallel to $B$ and $B'$, and so the claim is proved. 
 
 Otherwise,  
  $B$ is not parallel to  $B'$. 
  If $\ell$ is $B$ or $B'$, then $\ell$ clearly bisects $Q$, so we assume $\ell$ is neither of these lines. By Proposition~\ref{sides}(3), $\ell$ is not parallel to $B$ or $B'$ since these last two lines are not parallel to each other. Therefore, $\mu h+t_Bk \ne 0$ and $\mu h+t_{B'}k \ne 0$, and so  $\ell$ intersects $B$ and $B'$ in the points
$$\ell \cdot B = \left(-\frac{kv_B}{\mu h+t_Bk}, \frac{v_B \mu h}{\mu h+t_Bk}\right), \:\: \ell \cdot B' =
\left(-\frac{kv_{B'}}{\mu h+t_{B'}k}, \frac{v_{B'} \mu h}{\mu h+t_{B'}k}\right).$$
Equations  (\ref{centroid eq})  imply
 $(0,0)$ is the midpoint of these two points, and so the line $\ell$ 
  bisects $B,B'$ with midpoint $(0,0)$. This proves the line $kY+\mu hX=0$ bisects $Q$. 
 \end{proof}
 
 \begin{corollary} \label{vert} \label{brand new cor}  If $Q$ is a quadrilateral in standard form,
 then 
the centroid of $Q$ is the origin 
  if and only if the vertices of $Q$ are the vertices of a parallelogram.
  \end{corollary} 
 
 \begin{proof} 
If the vertices of $Q$ are the vertices of a parallelogram $P$, 
then since $Q$ is in standard form and $A$ and $A'$ are not parallel, it follows that $A$ and $A'$ are the diagonals of $P$ and hence the origin (the intersection of $A$ and $A'$) is the center of $P$, which is the centroid of $Q$ since $P$ and $Q$ share the same vertices.   
The converse follows from Propositions~\ref{unique} and~\ref{brand new2}(2).
\end{proof}



%

The last corollary shows one of the main advantages of reducing to standard form. 

\begin{corollary} \label{same same} If two quadrilaterals in standard form share the same coefficient and the same centroid, then   both quadrilaterals have the same bisectors. 
\end{corollary} 

\begin{proof} This follows from Proposition~\ref{brand new1} since the description of the bisectors in this proposition depends only on the centroid and  coefficient of the quadrilateral.
\end{proof}

The converse of the corollary is also true. This is proved in \cite{OW6} using different methods. 
   
 
 \section{The bisector locus}

In this section we describe the
 {\it bisector locus} of a quadrilateral $Q$ (or quadrangle $\Q$), which is  the locus of midpoints of the bisectors of $Q$ (or $\Q$).  By Proposition~\ref{two three}, the bisector locus of a quadrangle is the bisector locus of any quadrilateral belonging to the quadrangle. 
   For the sake of stating the  next theorem, we adapt the notion of a diagonal point of a quadrangle to that of a quadrilateral: A {\it diagonal point} of a quadrilateral $Q$ is a point that is the intersection of a pair of opposite sides of $Q$ or the diagonals of $Q$.  Since
   the quadratic form $\Phi_Q$ can be viewed as a squared norm for the inner product space defined in Definition~\ref{alpha def notation}
 and determined by $Q$, 
   the idea behind the next  theorem is that the bisector locus  is either a circle or a pair of lines in the geometry of $Q$.

\begin{theorem} \label{nine theorem} Let $Q$ be a quadrilateral, let $(h,k)$ be the centroid of $Q$ and let $(a,b)$ be a diagonal point of $Q$ that is not at infinity. 
 The bisector locus of   $Q$ is  given by   $$\Phi_Q(X-h,Y-k) = \Phi_Q(a-h,b-k).$$     Thus the bisector locus is a conic with center $(h,k)$. 
\end{theorem}

\begin{proof}    
Suppose first $Q$ is in standard form and $(a,b) =(0,0)$. By Lemma~\ref{standard phi}, $\Phi_Q(X,Y) = Y^2-\mu X^2$, and by Proposition~\ref{brand new1},  
the bisector locus is the zero set of the polynomial  $Y^2-2kY-\mu(X^2-2hX) $. Completing the square,  the bisector locus is defined by the equation $(Y-k)^2-\mu(X-h)^2=k^2-\mu h^2$.  
 Thus $\Phi_Q(X-h,Y-k)= \Phi_Q(h,k)$, which proves the theorem when $Q$ is in standard  form and $(a,b)$ is the origin.

 Now suppose $Q=ABA'B'$ is not necessarily in standard form.
  Using Proposition~\ref{two three}  we can switch quadrilaterals if necessary and assume that $A$ and $A'$ are not parallel and intersect at the point ${\bf a}=(a,b)$. 
  There is an affine transformation $f$ that carries ${\bf a}$ to the origin and carries $Q$ onto a quadrilateral in standard form. 
  Let $g$ be a  linear transformation such that $f({\bf x}) = g({\bf x}) -g({\bf a})$ for all ${\bf x} \in \k^2$.   
  Let ${\bf c}=(h,k)$ be the centroid of $Q$. Since affine transformations preserve midpoints and hence bisectors, a point ${\bf x}=(x,y) \in \k^2$ is on the bisector locus of $Q$ if and only if   $f({\bf x})$ is a point on the bisector locus of ${f}(Q)$. Let $\mu$ be the coefficient of the quadrilateral $f(Q)$. 
 Since $\Phi_{f(Q)}$ depends only on the slopes of the lines that comprise $f(Q)$, it follows that $ \Phi_{f(Q)} = \Phi_{g(Q)}$.  
   The centroid of ${f}(Q)$ is $f({\bf c})$ and   the theorem has been verified for quadrilaterals in standard form, so we use Proposition~\ref{image orthogonal} to obtain that 
 for each ${\bf x}$ in $\k^2$, 
 \begin{eqnarray*}
f({\bf x}) \in {\mbox{bisector locus }} & \Longleftrightarrow & \Phi_{f(Q)}(f({\bf x})- f({\bf c}))  =  \Phi_{f(Q)}(f({\bf a})-f({\bf c})) \\
 & \Longleftrightarrow &
 \Phi_{g(Q)}(g({\bf x}-{\bf c})) = \Phi_{g(Q)}(g({\bf a}-{\bf c})) \\
 &  \Longleftrightarrow &
  \Phi_Q({\bf x}-{\bf c}) = \Phi_Q({\bf a}-{\bf c}).
  \end{eqnarray*}
  This proves the theorem. 
 \end{proof}

   A conic over $\k$ is {\it degenerate} if the polynomial that defines it is a product of linear polynomials over the algebraic closure of $\k$. 
 Theorem~\ref{deg}, which is illustrated by Figure~3,  gives a criterion for when the bisector locus is degenerate. 

 \begin{theorem} \label{deg} 
 The bisector locus of a quadrilateral $Q$ is degenerate if and only if 
 there is a pair $L,L'$ of parallel sides or diagonals of $Q$. In this case, the bisector locus of $Q$ is the union of the midline of $L$ and $L'$ and  the line through the midpoints of $L$ and $L'$.
 \end{theorem} 
 
 \begin{proof}  
   Assume first $Q = ABA'B'$ is in standard form. By Lemma~\ref{standard phi},  $\Phi_Q(X,Y) = Y^2-\mu X^2$. 
Suppose the bisector locus is degenerate.
If $(h,k) = (0,0)$, then by Corollary~\ref{vert}, $Q$ has a pair of parallel sides or diagonals, so we may assume $(h,k) \ne (0,0)$. 
 By Theorem~\ref{nine theorem}, 
 $\Phi_Q(h,k) = k^2-\mu h^2 =0$, so 
by Proposition~\ref{brand new1}(1), 
for every $t \in \k$, 
since 
  $$tk(tk-2k) -\mu th(th-2h)=0, $$ we have that   
    $(th,tk)$ is the midpoint of a bisector given by 
   $tkX+thY-2t^2hk =0$. 
  Each such bisector is distinct (since $(h,k) \ne (0,0)$) and has the same slope, so by Proposition~\ref{sides}(3),  
  %
$Q$ has a pair of parallel sides or diagonals.

Conversely, still assuming $Q$ is in standard form, suppose there is a pair  of   sides or diagonals of $Q$ that are parallel. By Proposition~\ref{two three}  we can switch quadrilaterals if necessary in order to  assume $B$ and $B'$ are parallel and $Q$ is in standard form.  Let $\mu$ be the coefficient of $Q$. 
Since $t_B=t_{B'}$ and $t_Bt_{B'}=\mu$, we have $t_B^2=\mu$. Also, equation~(\ref{centroid eq}) from the proof of Proposition~\ref{brand new1}(1) implies $-4h\mu = t_B(v_B+v_{B'}) = 4t_Bk$, and so $k = -t_Bh$.   
Lemma~\ref{standard phi} and
Theorem~\ref{nine theorem}  imply then that
the bisector locus of $Q$ is the zero set of  $$Y(2k-Y) -\mu X(2h-X)= 
(t_B X+Y)(t_B X-Y+2k).$$ The  line
 $Y=-t_B X $ is the line  through the midpoints of the parallel sides $B$ and $B'$ and 
 the 
 line $Y=t_B X+2k$ is the 
   midline of $B$ and $B'$,
  so the   theorem is proved if $Q$ is in standard form.
 
Now suppose $Q=ABA'B'$ is not necessarily in standard form. By Proposition~\ref{two three}, we can assume $A$ and $A'$ are not parallel and hence  meet at a point $(a,b)$. 
 Let ${\bf c} = (h,k)$ and  ${\bf a}=(a,b)$.  
By Theorem~\ref{nine theorem}, the bisector locus of $Q$ is degenerate if and only if $\Phi_Q({\bf a}-{\bf c}) =0$.  
 Let $g$ be a linear transformation such that the map $f$ from the plane to itself defined by  $f({\bf x}) = g({\bf x}) - g({\bf a})$ for all ${\bf x} \in \k^2$  is an affine transformation that carries $Q$ onto a quadrilateral in standard form and ${\bf a}$ to the origin. Since $f(Q)$ is in standard form, $f(Q)$, and hence $Q$, has a pair of parallel sides or diagonals if and only if $\Phi_{f(Q)}(f({\bf a}) -f({\bf c})) =0$. Since $\Phi_{f(Q)}$ depends only on the slopes of the lines that comprise $f(Q)$, the quadrilateral $Q$ has a pair of parallel sides or diagonals 
 if and only if $\Phi_{g(Q)}(g({\bf a}-{\bf c}))=0$; if and only if $\Phi_Q({\bf a}-{\bf c}) =0$, where this last equivalence follows from Proposition~\ref{image orthogonal}. 
     \end{proof}

 The {\it nine-point conic} of a quadrangle $\Q$ is the unique conic passing through the six {midpoints} of the quadrangle, i.e., the midpoints of each pair of vertices,
 and the three {diagonal points} of $\Q$, the points at which   opposite sides meet, where possibly these diagonal points are at infinity. 
 For a  proof of the existence of the nine-point conic over an arbitrary field of characteristic $\ne 2$,  see \cite[16.5.5.1, p.~198]{Berger}.  
  In our case, the existence   follows from the fact that the bisector locus is a conic that passes through these same nine points.

 \begin{corollary} \label{nine} The bisector locus of a quadrangle  $\Q$ is the  conic through the diagonal points and midpoints of $\Q$ and hence is the nine-point conic for $\Q$. 
\end{corollary}

 \begin{proof} 
The six sides of $\Q$ are bisectors with respect to their midpoints, and so the midpoints of these  sides lie on the bisector locus. 
We claim the diagonal points of $\Q$ are also midpoints of bisectors of~$\Q$.  If a side $A$ of $\Q$ is parallel to its opposite side $A'$, then the diagonal point for $A$ and $A'$ is the point at infinity for these lines. In this case,  Theorem~\ref{deg} implies the midline of  $A$ and $A'$ is part of the bisector locus. Since the point at infinity for the midline is the diagonal point of $A$ and $A'$, this point is a point at infinity for the bisector locus of~$\Q$.  

Now suppose $A$ is not parallel to $A'$. If the other two pairs of opposite sides of $\Q$ form a parallelogram, then $A$ and $A'$ are the diagonals of this parallelogram and these lines meet at the center of the parallelogram, in which case the diagonal point $A \cdot A'$ is the side midpoint of both $A$ and $A'$ and hence the midpoint of a bisector. Otherwise, if there is another pair  $B,B'$ of opposite sides that are not parallel,  
then $A \cdot A'$ is the midpoint of a point  on $B$ and a point  on $B'$, and hence  
the line through these two points is  a bisector of $ABA'B'$ having $A \cdot A'$ as its midpoint.  By Proposition~\ref{two three}, this line bisects $\Q$ and so $A \cdot A'$ is on the bisector locus of $\Q$.   
%
\end{proof}

 \section{The bisector field  of a quadrilateral}

By a {conic} in the affine plane over $\k$, we mean a  quadratic polynomial  in $\k[X,Y]$. 
A {hyperbola}  is a centered  conic that has two distinct points at infinity. Its {asymptotes}   are the two lines through the center of the hyperbola that meet the hyperbola at infinity. 
A {parabola} is a conic with one point at infinity and an {ellipse} a conic with no points at infinity. A pair of parallel lines or a double line is  a degenerate parabola.

A pair of lines in the affine plane is a
 if there is  $\lambda \in \k$ such that  the union of these two  lines is the set of zeroes of $f +\lambda$.  
Neither an ellipse nor a parabola  has a degeneration, and  
 a hyperbola has a unique degeneration, namely its asymptotes. A 
 pair of parallel lines $\ell,\ell'$  has as its degenerations the pairs of lines parallel to $\ell$ and $\ell'$ that share the same midline as the pair  $\ell,\ell'$.  See \cite[Proposition 2.2]{OW4} for more details.


 
 To each quadrilateral $Q=ABA'B'$, we associate  a pencil   of conics. Let $f$ and $g$ be degenerate 
conics whose zero sets are $A \cup A'$ and $B \cup B'$, respectively.  
The {\it pencil of} 
$Q$ is  
$$
 {\rm Pen}(Q) \: = \: \{\alpha f_1 + \beta f_2:\alpha,\beta   \in \k {\mbox { \rm with }}  \alpha { \mbox { \rm and }} \beta   { \mbox { \rm not both zero}}\} 
$$
   If $Q$ is a proper quadrilateral, then the pencil of $Q$  is simply the set of conics passing through all four vertices of $Q$. If instead $Q$ is an  improper quadrilateral and $p$ is the double vertex of $Q$, the vertex through which three sides   pass, say $A,A',B$, then the pencil of $Q$ is the set of conics that pass through all three vertices and   are tangent to $B$ at $p$. 
The {\it asymptotic pencil} of $Q$ is the set of reducible conics of the form $f+\lambda$, where $\lambda \in \k$ and $f \in $ Pen$(Q)$; i.e., the asymptotic pencil of $Q$ is the set of conics whose zero sets are the degenerations of the conics in Pen$(Q)$.    It consists of the asymptotes of the hyperbolas in Pen($Q$) and the pairs of parallel lines that share their midline with that of a pair of parallel lines in Pen($Q$), if any such parallel lines exist. 
See Figure~4 for an illustration. 
The asymptotic pencil is  the subject of \cite{OW4}.

   \begin{figure}[h] 
     \label{9noobconics}
 \begin{center}
\includegraphics[width=.95\textwidth,scale=.09]{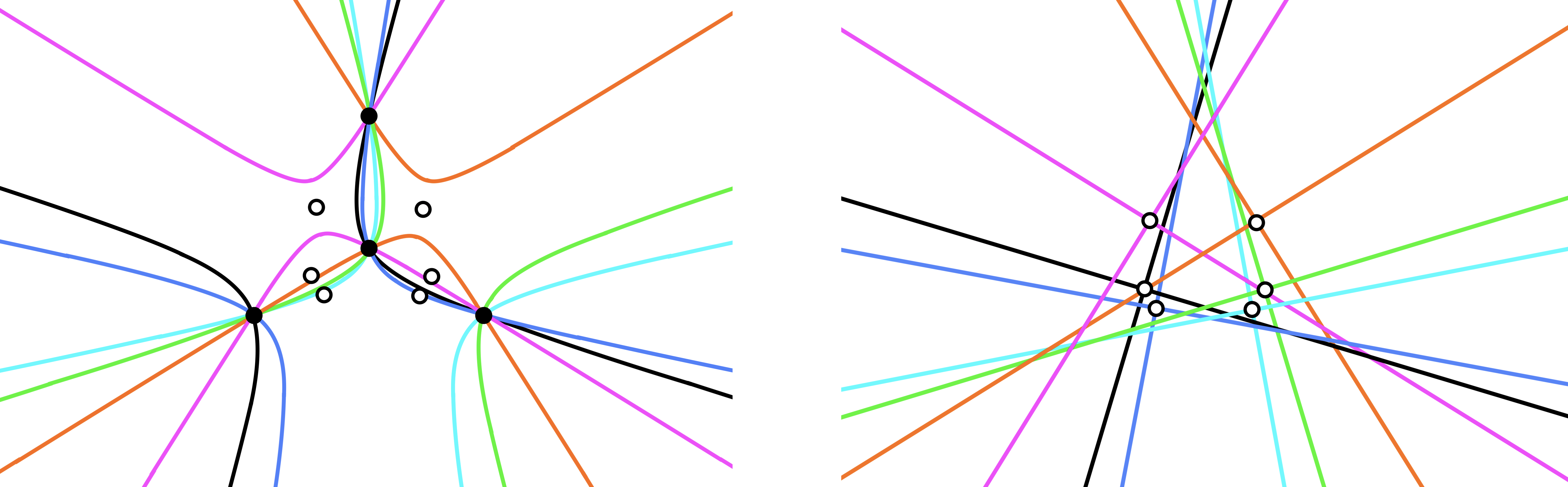} 
 \end{center}
 \caption{On the left are several conics belonging to the pencil through the four 
  black points. 
  The centers of these conics are marked with white points. On the right are the asymptotes of these same conics.
  The asymptotes of the  black hyperbola at left are the black lines at right.   By Theorem~\ref{degenerations},  these pairs belong to the bisector field of the quadrilateral defined by the pencil. The midpoints of the bisectors are the white points.
     }
\end{figure}

The next lemma, which is crucial in the proofs of several of the results that follow, gives a first connection between bisectors and asymptotic pencils. 

\begin{lemma}  \cite[Lemma~5.2]{OW4} \label{OW4 lemma}
A line $\ell$ is a bisector of a quadrilateral  $Q$ if and only if it is part of a degeneration of a conic in the pencil of $Q$.  
\end{lemma}

Thus if $\ell$ is a bisector of a quadrilateral $Q$,  there is another bisector $\ell'$ of $Q$ such that $\ell,\ell'$ is  a degeneration of a conic in Pen$(Q)$.  This suggests a natural pairing on the bisectors of $Q$,
one that proves essential  in the context of bisector fields, which are considered at the end of this section. 
There are, however, two other natural pairings on the bisectors of $Q$, that of $Q$-orthogonality and another involving antipodal points. 
Because of some pathologies involving parallel lines, $Q$-orthogonality itself is not quite sufficient to handle all cases, so we combine this relationship with an antipodal one across the centroid of $Q$. 

  \begin{definition} \label{Q pair def}
   A pair   of bisectors of the quadrilateral $Q$ is 
    {\it $Q$-antipodal} if
the midpoint of the midpoints of the bisectors   is the centroid of $Q$. 
The pair is a
 {\it $Q$-pair} 
if  it is  $Q$-antipodal and $Q$-orthogonal. 
%
 \end{definition}

We will show in Corollary~\ref{orthogonal pair} that in most cases, a pair of bisectors is    $Q$-antipodal if and only if it is $Q$-orthogonal. This is not true in all cases, however, as the 
following example shows.
 
 \begin{example} \label{parallel remark} The bisector locus of a parallelogram $Q$ is a pair of lines meeting at the center of $Q$ and parallel to the sides of $Q$, as in Figure~5. It is straightforward to verify that the $Q$-pairs of bisectors of $Q$ are
(a) the pairs of lines parallel to a pair of sides of $Q$ and having the same midline as these sides,  
 (b) the midlines of the parallel pairs in (a), where each midline is paired with itself, and   
 (c) the diagonals of the parallelograms whose pairs of opposite sides are the pairs in (a). 
 Where the lines in (a) cross a midline from (b) are the midpoints of these bisectors, while  
 the bisectors in (b) and (c) all have the centroid of $Q$ as their midpoint. 
Any pair of lines parallel to a midline in (b) is a $Q$-orthogonal pair of bisectors of $Q$ but need not be a $Q$-antipodal pair. Similarly, any pair of lines through the center of $Q$ is a $Q$-antipodal pair of bisectors that need not be $Q$-orthogonal. 
 \end{example}

      \begin{figure}[h] 
     \label{9noobconics}
 \begin{center}
\includegraphics[width=.65\textwidth,scale=.09]{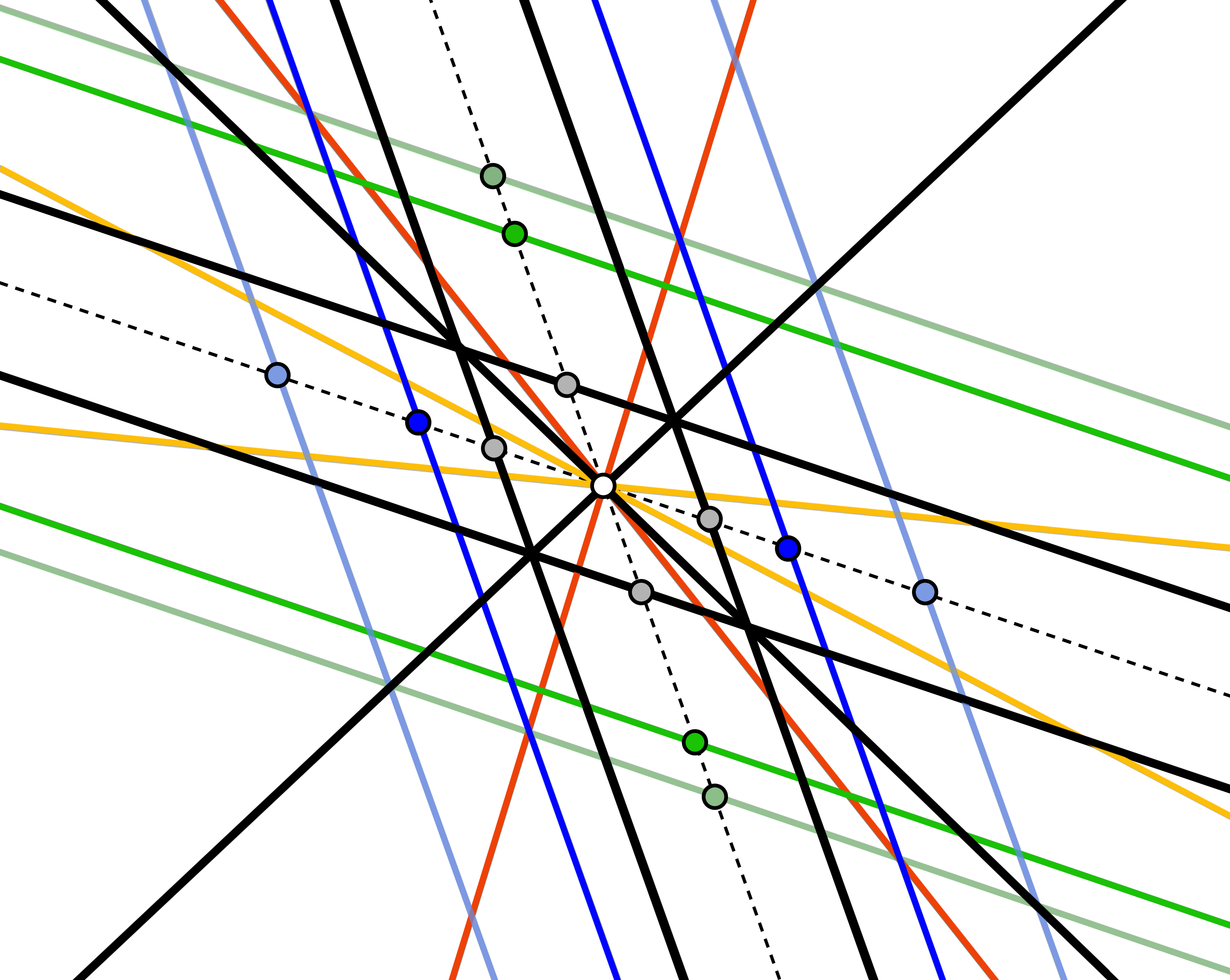} 
 \end{center}
 \caption{Pairs from a bisector field of  a parallelogram. All the lines through the center point are bisectors having the center as their  midpoint. The lines   through the center are paired according to the shade of gray with which they are colored. The parallel lines are paired via antipodal midpoints.    }
\end{figure}

The next theorem gives a quadrilateral-specific way to detect the degenerations of the conics in the pencil of a quadrilateral  
that is independent of reference to the pencil itself.

\begin{theorem} \label{degenerations}  
The $Q$-pairs of bisectors of $Q$ are precisely the 
degenerations of the conics in the pencil of $Q$. 
%
\end{theorem}

 \begin{proof}  
By Proposition~\ref{two three}, if $Q$ is a parallelogram we may switch to a quadrilateral that has the same vertices as $Q$ but for which $Q$ is not a parallelogram. Thus, after an affine transformation, 
 we may assume  $Q=ABA'B'$   is in standard form.   
Let  $v = v_Bv_{B'}$,  let $(h,k)$ be the centroid of $Q$ and let $\mu = t_Bt_{B'}$ be the coefficient of $Q$. 
 Using equations (\ref{vertices}) and (\ref{centroid eq}) from the proof of Proposition~\ref{brand new1}(1),  
   along with the fact that $t_Bt_{B'}=\mu$,  it
   follows that 
 the conics in the pencil of $Q$ are  given by the equations of the form 
 \begin{eqnarray} \label{first conic}  \mu a X^2 + \mu bXY + a Y^2 -4a\mu h X - 4akY +a v=0.
 \end{eqnarray} 
 where $a$ and $b$ range over the elements in $\k$ with   $a$ and $b$  not both $0$. 
 
We first prove that if $\ell_1,\ell_2$ is a degeneration of one of the conics of the form (\ref{first conic}), and at least one of $\ell_1,\ell_2$ is parallel to the $X$ or $Y$-axis, then $\ell_1,\ell_2$ is a $Q$-pair of bisectors of $Q$. 
 Let $\ell_1,\ell_2$ be a pair of lines that is a degeneration of a conic in  (\ref{first conic}). 
By Lemma~\ref{OW4 lemma}, $\ell_1$ and $\ell_2$ are bisectors of $Q$.   
 Suppose first that one of these lines is parallel to the $Y$-axis. Then there are $\lambda,a,b,t,u,v,w \in \k$ such that at least one of $a,b$ is nonzero, at least one of $t,u$ is nonzero and 
 $$\mu a X^2 + \mu bXY + a Y^2 -4a\mu h X - 4akY +a v +\lambda = (X+w)(tX-uY+v).$$  This implies $a =0$ and hence $t =v=w =0$, so that $\ell_1$ and $\ell_2$ are the $X$ and $Y$-axes. Thus $\ell_1,\ell_2$ is a $Q$-pair of bisectors since the axes are a pair of opposite sides of $Q$. A similar argument 
 shows that if one of $\ell_1$ and $\ell_2$ is parallel to the $X$-axis, then $a =0$ in equation (\ref{first conic}), and $\ell_1,\ell_2$ are the axes and hence a $Q$-pair of bisectors of $Q$. 
 
 
 
We can therefore reduce to the case that 
 $a=1$ in equation (\ref{first conic}). The next thing to prove is that  if $\ell_1,\ell_2$ is a degeneration of a conic of the form $$f_b(X,Y) := \mu  X^2 + \mu bXY +  Y^2 -4\mu h X - 4kY + v, {\mbox{ where }} b \in \k,$$ such that neither $\ell_1$ nor $\ell_2$ is parallel to an axis, 
 then $\ell_1,\ell_2$ is a $Q$-pair of bisectors of $Q$. 
 Let $\ell_1,\ell_2$ be a
degeneration of $f_b$ for some $b \in \k$, where neither $\ell_1$ nor $\ell_2$ is parallel to an axis. Then $\ell_1$ and $\ell_2$ are bisectors of $Q$ by Lemma~\ref{OW4 lemma}, and 
there are $\lambda,t_1,v_1,t_2,v_2 \in \k$ such that 
  $$f_b(X,Y) +\lambda = (t_1X-Y+v_1)(t_2X-Y+v_2).$$ 
 Equating coefficients yields $$\mu = t_1t_2 \quad -4\mu h = t_1v_2+t_2v_1 \quad 4k = v_1+v_2.$$  
 Since $\mu = t_1t_2$, the pair  $\ell_1,\ell_2$ is  $Q$-orthogonal. To see that $\ell_1,\ell_2$ is a $Q$-antipodal pair, consider the quadrilateral $Q' = A\ell_1A'\ell_2$. Its vertices are $(0,v_1),(0,v_2), (-v_1/t_1,0),(-v_2/t_2,0)$. The centroid of $Q'$, which is the mean of these vertices, is $$\left(\frac{-v_1t_2-v_2t_1}{4t_1t_2}, \frac{v_1+v_2}{4}\right)=
 \left(\frac{-v_1t_2-v_2t_1}{4\mu}, \frac{v_1+v_2}{4}\right) =
 (h,k).$$ The midpoints of $\ell_1$ and $\ell_2$ as bisectors of $Q$ are the same as the midpoints of $\ell_1$ and $\ell_2$ as sides of $Q'$, so since the midpoint of the midpoints of a pair of opposite sides of a quadrilateral is the centroid of that quadrilateral, we conclude that $\ell_1,\ell_2$ is a $Q$-antipodal pair.  In all cases, we have proved that if $\ell_1,\ell_2$ is a degeneration of a conic in the pencil of $Q$, then $\ell_1,\ell_2$ is a $Q$-pair of bisectors of $Q$.  
 
Conversely, suppose  $\ell_1,\ell_2$ is a $Q$-pair of bisectors of $Q$. We show $\ell_1,\ell_2$ is a degeneration of a conic  in the pencil of $Q$.  Lemma~\ref{OW4 lemma} implies that since $\ell_1$ is a bisector of $Q$, there is a line $\ell_2'$ such that $\ell_1,\ell_2'$ is a degeneration of a conic in the pencil of $Q$. 
By what we have established, $\ell_1,\ell_2'$ 
is therefore a $Q$-pair of bisectors. Thus $\ell_1$ is $Q$-orthogonal to $\ell_2$ and $\ell_2'$, and so $\ell_2$ and $\ell_2'$ are parallel. Also, since $\ell_1,\ell_2$ and $\ell_1,\ell_2'$ are $Q$-antipodal pairs, $\ell_2$ and $\ell_2'$ share the same midpoint, so $\ell_2 = \ell_2'$.  
%
Therefore, 
 $\ell_1,\ell_2$ is a degeneration of a conic in the pencil of~$Q$. 
 \end{proof}

The theorem implies that the  $Q$-pairs of bisectors of $Q$ that are not parallel to each other are precisely the 
asymptotes of the hyperbolas in the pencil of $Q$. 
The point of the next corollary is that for almost all choices of quadrilaterals and pairs of bisectors, the properties of being $Q$-antipodal and $Q$-orthogonal are redundant. 

 

%



 \begin{corollary} \label{orthogonal pair} Let $Q$ be a quadrilateral. 
 \begin{enumerate}
 \item[$(1)$] 
  Every $Q$-orthogonal pair of  bisectors not both parallel to two sides or diagonals of $Q$ is  $Q$-antipodal and hence is a $Q$-pair. 
   
   \item[$(2)$] 
    Every $Q$-antipodal pair of bisectors that do not share a midpoint is   $Q$-orthogonal and hence is a $Q$-pair.
    \end{enumerate}  

 \end{corollary}
 
 \begin{proof} 
 %
 (1) Suppose $\ell_1,\ell_2$ is a $Q$-orthogonal pair of bisectors of $Q$ that are not parallel to a pair of sides or diagonals of $Q$. By Lemma~\ref{OW4 lemma}, there is a bisector $\ell_2'$   such that $\ell_1,\ell_2'$ is a degeneration  of a conic in the pencil of $Q$.  
 By Theorem~\ref{degenerations}, $\ell_1,\ell_2'$ is a $Q$-pair.  
 Since $\ell_1,\ell_2$ is a $Q$-orthogonal pair of bisectors, $\ell_2$ and $\ell_2'$ are parallel. If   $\ell_2 \ne \ell_2'$, then by Proposition~\ref{sides}(3), $\ell_2$ is  parallel to a pair of sides or diagonals of $Q$, contrary to assumption. Thus   $\ell_2 = \ell_2'$, and so $\ell_1,\ell_2$ is a degeneration of a conic in the pencil of $Q$. By Theorem~\ref{degenerations}, $\ell_1,\ell_2$ is a $Q$-pair.
 
 (2) Suppose $\ell_1,\ell_2$ is a $Q$-antipodal  pair of bisectors that do not share a midpoint. As in the proof of (1),  there are bisectors $\ell_1'$ and $\ell_2'$ such that $\ell_1,\ell_2'$ and $\ell_1',\ell_2$ are  $Q$-pairs of bisectors. Since these pairs are $Q$-antipodal, $\ell_2'$ and $\ell_2$ share the same midpoint, as do $\ell_1$ and $\ell_1'$.  By assumption, $\ell_1$ and $\ell_2$ do not share a midpoint, so by Proposition~\ref{unique}, either $\ell_1 = \ell_1'$ or $\ell_2 = \ell_2'$. As in the proof of (1), this implies $\ell_1,\ell_2$ is a $Q$-pair. 
 \end{proof} 
 
 Whether two bisectors are $Q$-orthogonal is {\it a priori} not easy to determine  geometrically since orthogonality depends on the inner product defined in Section~$2$. The next corollary, however,  shows that for  a generically chosen quadrilateral, there is a simple geometric way to distinguish whether two bisectors are $Q$-orthogonal. 

\begin{corollary}     If the quadrilateral $Q$ has no parallel sides or diagonals, then a pair of bisectors  of $Q$ is $Q$-antipodal if and only if it is $Q$-orthogonal. 
 \end{corollary} 
 
 \begin{proof} Apply Proposition~\ref{unique} and Corollary~\ref{orthogonal pair}. 
 \end{proof}

 The following corollary can be compared to the more general setting in \cite{OW4}, where quadrilaterals can have a vertex at infinity. In that case, bisectors may not have unique partners; see \cite[Lemma 6.2]{OW4}. 
  
\begin{corollary} 
 \label{existence}  If $\ell$ is a bisector of a quadrilateral $Q$, then there is a unique bisector $\ell'$ for which $\ell,\ell'$ is a $Q$-pair. \end{corollary} 

\begin{proof} 
Let $(p,q)$ be the midpoint of $\ell$, and  let $(p',q')$ be the point  that is antipodal to $(p,q)$ across the centroid of $Q$. If $(p,q) \ne (p',q')$, then 
by Proposition~\ref{unique} and the fact that the bisector locus is a centered conic (Theorem~\ref{nine theorem}), there is a unique bisector $\ell'$ with midpoint $(p',q')$, and  so by Corollary~\ref{orthogonal pair}, $\ell,\ell'$ is a $Q$-pair and $\ell'$ is the only bisector of $Q$ with this property.

Otherwise, if  
 $(p,q) = (p',q')$, then $(p,q)$ is the centroid of $Q$ since $\ell,\ell'$ is a $Q$-antipodal pair. 
 There is a unique line $\ell'$  through $(p',q')$ that is $Q$-orthogonal to $\ell$. If $\ell = \ell'$, then clearly $\ell'$ is a bisector; if $\ell \ne \ell'$, then $\ell'$ is a bisector  with midpoint $(p,q)=(p',q')$ by Proposition~\ref{unique} and the discussion in Example~\ref{parallel remark}.   The line $\ell'$  is the unique bisector of $Q$ such that the pair $\ell,\ell'$ is a $Q$-pair. 
\end{proof}


In \cite{OW4} it is shown that each line in an asymptotic pencil of a quadrilateral bisects every pair of lines in the asymptotic pencil with respect to the same midpoint. (In fact, each such line  also bisects  the conics in the pencil of $Q$ \cite[Theorem 5.3]{OW4}.) 
 This is expressed in \cite{OW4} using the notion of a  bisector field. 

 \begin{definition} A collection $\B$  of pairs of lines is a {\it bisector arrangement} if each line in each pair in $\B$  bisects $\B$. A bisector arrangement is {\it trivial}   if 
all lines in the arrangement share the same point (possibly at infinity) or every pair is a translation of every other pair. 
A  {\it bisector field} is a nontrivial bisector arrangement that  
 cannot be extended to a larger bisector arrangement. The bisector field is {\it affine} if no line in one pair is parallel to a line in another pair. 
\end{definition} 

The reason for ``affine'' in the  definition of affine bisector field is that these are the  bisector fields for which no bisector  in the bisector field has a midpoint  at infinity. 
 In \cite{OW4}, a more general definition of  quadrilateral  than the present one is allowed, one in which adjacent sides can  be parallel. Say that a arrangement of four lines is an {\it infinite quadrilateral} if it is a quadrilateral as defined in the Introduction, except that two adjacent sides are  parallel.  
 This quadrilateral is ``infinite'' because it  has a vertex at infinity.  
 One of the main results of \cite{OW4} is that a collection of a paired lines is a nontrivial asymptotic pencil if and only if it is a bisector field; see \cite[Theorem 6.3]{OW4}. 
 In this case the asymptotic pencil is generated by a  quadrilateral or an infinite quadrilateral \cite[Proposition 4.5]{OW}. 
We use this to deduce the next theorem for our slightly  more restrictive setting.

 \begin{theorem}  \label{ultimate}  
  \label{is a bisector field}
  A collection $\B$ of paired lines  is an affine bisector field if and only if 
 these line pairs are the $Q$-pairs of bisectors of a quadrilateral $Q$. In particular, 
every bisector of $Q$ bisects every $Q$-pair of bisectors of $Q$.

  %
 \end{theorem} 
 
 \begin{proof}  
 Suppose $\B$ is an affine bisector field. By  \cite[Theorem 6.3]{OW4}, $\B$ is the set of degenerations of conics in a pencil of a possibly infinite quadrilateral $Q$. 
 Since the pairs of opposite sides of $Q$ are in $\B$ and $\B$ is affine, $Q$ is a quadrilateral as defined in the Introduction since in this case adjacent sides of $Q$ cannot be parallel. By Theorem~\ref{degenerations}, $\B$  is the set of $Q$-pairs of bisectors of $Q$.  
 Conversely, 
suppose $\B$ is the set of $Q$-pairs of bisectors of $Q$. Then $\B$ is the set of degenerations of the conics in the pencil  of $Q$ by Theorem~\ref{degenerations}.  By \cite[Theorem 6.3]{OW4}, $\B$ is a bisector field. If $\B$ is not affine, then $\B$ contains two pairs $A,A'$ and $B,B'$ with $A$ parallel to $B$ but not parallel to $B'$.  The midpoint of $A$ as a bisector of $B,B'$ is therefore at infinity, and since $\B$ is a bisector field, this implies $A$ is  parallel to at least one line in every pair in $\B$. But then $Q$, as a quadrilateral in $\B$, must have a pair of adjacent sides that are parallel to $A$, contrary to the assumption that $Q$ is a (finite) quadrilateral. 
 Therefore, $\B$ is affine. 
%
 \end{proof} 
 


%
By \cite[Theorem 5.3]{OW4}, each bisector $\ell$ of $Q$  also bisects  every conic in the pencil of $Q$. When  $Q$ is a proper quadrilateral, this last assertion can also be derived from Desargues' Involution Theorem (Lemma~\ref{des}).

The collection of $Q$-pairs of bisectors of a quadrilateral $Q$ is the {\it bisector field of} $Q$. 


\begin{corollary} \label{same bisectors} Two quadrilaterals   have the same bisectors  if and only if they have the same bisector fields.
\end{corollary} 

\begin{proof} 
Suppose $Q_1$ and $Q_2$ are quadrilaterals that have the same bisectors. We show $Q_1$ and $Q_2$ have the same bisector fields. It suffices by Theorem~\ref{ultimate} to show that if 
 $\ell,\ell'$ is a $Q_1$-pair of bisectors of $Q_1$, then $\ell,\ell'$ is $Q_2$-pair. 
 Since $Q_1$ and $Q_2$ have the same bisectors, they have the same bisector locus and hence the same centroid.  Thus a pair $\ell,\ell'$  of bisectors  of $Q_1$ is $Q_1$-antipodal if and only if it is $Q_2$-antipodal. If the midpoints of $\ell$ and $\ell'$ are distinct, then
by Corollary~\ref{orthogonal pair}
 this implies $\ell,\ell'$ is a $Q_1$-pair if and only if it is a $Q_2$-pair. 
 
 Otherwise, suppose $\ell,\ell'$ is a $Q_1$-pair of bisectors of $Q_1$ with the same midpoint.  Since these bisectors are $Q_1$-antipodal,   this shared midpoint is the centroid of $Q_1$.  
  By Corollary~\ref{existence},
 there is a bisector $\ell'' $ of $Q_2$ such that $\ell,\ell''$ is a $Q_2$-pair, and so since $\ell,\ell''$ is $Q_2$-antipodal and the centroid of $Q_1$ is the centroid of $Q_2$, it follows that $\ell$ and $\ell''$ have the centroid as a shared midpoint. If $\ell' = \ell''$, the  
  claim is proved, so suppose $\ell' \ne \ell''$. 
By Proposition~\ref{unique}, the vertices of $Q_1$ are the vertices of a parallelogram, and similarly for $Q_2$, since in both cases there are two distinct bisectors with the same midpoint.  
 The description   in  Example~\ref{parallel remark} implies that 
 since $Q_1$ and $Q_2$ have the same bisector locus, $\ell,\ell'$ is a $Q_1$-pair if and only if $\ell,\ell'$ is a $Q_2$-pair. This proves 
that $Q_1$ and $Q_2$ have the same bisectors if and only if they have the same bisector fields.   
\end{proof}


\begin{corollary} The only pairing that makes the bisectors of a quadrilateral $Q$ a bisector field is $Q$-pairing. 
\end{corollary}

\begin{proof} Apply Theorem~\ref{ultimate} and Corollary~\ref{same bisectors}.
\end{proof}

In a future paper \cite{OW6}  we  classify  bisector fields in terms of envelopes of lines tangent to curves.

\end{document}